\documentclass[reqno,12pt]{article}

\usepackage{a4wide}
\usepackage{amsmath} 
\usepackage{amssymb}
\usepackage{amsthm}
\usepackage[utf8]{inputenc} 
\usepackage{graphicx} 
\usepackage[english, russian]{babel}

\numberwithin{equation}{section}

\newtheorem{theo}{Theorem}

\newtheorem{prop}{Proposition}

\newtheorem{lem}{Lemma}

\theoremstyle{remark}

\newcommand{\tra}{{}^t}
\newcommand{\N}{\mathbb{N}}

\newcommand{\Q}{\mathbb{Q}}

\renewcommand{\C}{\mathbb{C}}
\newcommand{\K}{\mathbb{K}}
\renewcommand{\L}{\mathbb{L}}

\newcommand{\Qbar}{\overline{\mathbb Q}}

\newcommand{\etoile}{^*}

\newcommand{\GB}{Gr\"obner }
\newcommand{\Tsoul}{{\underline{T}}}
\newcommand{\asoul}{{\underline{a}}}
\newcommand{\bsoul}{{\underline{b}}}
\newcommand{\csoul}{{\underline{c}}}
\newcommand{\Xsoul}{{\underline{X}}}
\newcommand{\asj}{{\underline{a}^{(j)}}}
\newcommand{\asjindicei}{{\underline{a}_i^{(j)}}}
\newcommand{\gsoul}{{\underline{g}}}

\renewcommand{\Im}{{\rm Im}}
\newcommand{\lcm}{{\rm lcm}}
\newcommand{\specialll}{l}
\newcommand{\lmon}[1]{{\rm \specialll mon}(#1)}
\newcommand{\lexp}[1]{{\rm \specialll exp}(#1)}
\newcommand{\lcoeff}[1]{{\rm \specialll coeff}(#1)}
\newcommand{\lterm}[1]{{\rm \specialll term}(#1)}
\newcommand{\numlcoeff}[1]{{\rm num \,  \specialll coeff}(#1)}
\newcommand{\num}[1]{{\rm num}(#1)}

\newcommand{\unN}{\{1,\ldots,N\}}
\newcommand{\recoit}{:=}
\newcommand{\alp}{\alpha}
\newcommand{\iti}{\widetilde \imath}

\newcommand{\Iti}{\widetilde I}
\newcommand{\Wend}{W_{{\rm end}}}
\newcommand{\calP}{\mathcal P}
\newcommand{\calR}{\mathcal R}

\newcommand{\calF}{\mathcal F}

\newcommand{\calB}{\mathcal B}
\newcommand{\chialp}{\chi_\alp}
\newcommand{\Pun}{P^{[1]}}
\newcommand{\Pj}{P^{[j]}}
\newcommand{\Ps}{P^{[s]}}

\begin{document}

 \selectlanguage{english}

\title{Effective algebraic independence of values of $E$-functions}
\date\today
\author{S. Fischler and T. Rivoal}
\maketitle

\begin{abstract}
$E$-functions are entire functions with algebraic Taylor coefficients satisfying
certain arithmetic conditions, and which are also solutions of linear differential equations with
coefficients in $\Qbar(z)$. They were introduced by Siegel in 1929 to generalize the Diophantine properties of the exponential and Bessel's functions. The 
Siegel-Shidlovskii Theorem (1956) deals with the algebraic (in)dependence of values at algebraic
points of $E$-functions solutions of a differential system. In this paper, we prove the existence of an algorithm to perfom the following three tasks. Given as inputs some $E$-functions $F_1(z)$, \ldots, $F_p(z)$, 

$(1)$ it computes a system of generators of the ideal of polynomial relations between $F_1(z)$, \ldots, $F_p(z)$ with coefficients in $\Qbar(z)$; 

$(2)$ given any $\alpha\in\Qbar$,  it computes a system of generators of the ideal of polynomial relations between the values  $F_1(\alpha)$, \ldots, $F_p(\alpha)$ with coefficients in $\Qbar$;

$(3)$ if $F_1(z)$, \ldots, $F_p(z)$ are algebraically independent over $\Qbar(z)$,  it determines the finite set of all $\alpha\in\Qbar$ such that the values 
  $F_1(\alpha)$, \ldots, $F_p(\alpha)$ are algebraically dependent over $\Qbar$. 
  
  The existence of this algorithm  relies on a variant of the Hrushovski-Feng algorithm (to compute polynomial relations between solutions of differential systems) and  on Beukers' lifting theorem (an optimal refinement of the Siegel-Shidlovskii theorem) in order to reduce the problem to an effective elimination procedure in multivariate polynomial rings. The latter is then performed using \GB bases.
\end{abstract}

\section{Introduction}

A power series $F(z)=\sum_{n=0}^{\infty} \frac{a_n}{n!} z^n \in \Qbar[[z]]$ is an $E$-function  if  
\begin{enumerate}
\item[$(i)$] $F(z)$ is solution of a non-zero linear differential equation with coefficients in 
$\Qbar(z)$.
\item[$(ii)$] There exists $C>0$ such that for any $\sigma\in \textup{Gal}(\Qbar/\mathbb Q)$ and any $n\ge 0$,  $\vert \sigma(a_n)\vert \leq C^{n+1}$.
\item[$(iii)$] There exists $D>0$ and a sequence of integers $d_n$, with $1\le  d_n  \leq D^{n+1}$, such that
$d_na_m\in \mathcal{O}_{\Qbar}$ for all~$m\le n$.
\end{enumerate}
Above and below, we fix an embedding of $\Qbar$ into $\mathbb C$.
Siegel introduced in 1929 the notion of $E$-function as a generalization of the exponential and Bessel functions. His definition was in fact slightly more general than above (see the end of this introduction). Note that  $(i)$ implies that the $a_n$'s all lie in a certain   number field $\K$, so that in $(ii)$ there are only finitely many Galois conjugates $\sigma(a_n)$ of $a_n$ to consider, with  $\sigma\in \textup{Gal}(\mathbb K/\mathbb Q)$ (assuming for simplicity that $\K$ is a Galois extension of $\Q$).  An $E$-function  is transcendental over $\mathbb C(z)$ if and only if $a_n\neq 0$ for infinitely many $n$. For more informations about $E$-functions, we refer the reader to the survey \cite{rivoalsurvol}.

 Siegel proved in \cite{sieg29} a result on the Diophantine nature of the values taken by Bessel functions  at algebraic points. He generalized it to $E$-functions in 1949 in \cite{sieg} under a technical hypothesis ({\em Siegel's normality}), which  was eventually removed  by Shidlovskii in 1959, see~\cite{shid59}.  
\begin{theo}[Siegel-Shidlovskii]\label{thm:sish}
Let $Y(z)=\tra(F_1(z), \ldots, F_n(z))$ be a vector of $E$-functions such that 
$Y'(z)=A(z)Y(z)$ where $A(z)\in M_n(\Qbar(z)).$
Let $T(z)\in \Qbar[z]\setminus\{0\}$ be such that $T(z)A(z)\in M_n(\Qbar[z])$. 
Then for any $\alpha\in \Qbar$ such that $\alpha T(\alpha)\neq 0$, 
$$
\textup{degtr}_{\Qbar} \Qbar(F_1(\alpha), \ldots, F_{n}(\alpha)) =\textup{degtr}_{\Qbar(z)}\Qbar(z)(F_1(z), \ldots, F_{n}(z)).
$$
\end{theo}
The next step was the following result~\cite{ns} which essentially says that a numerical polynomial relation between values of $E$-functions   at an algebraic point cannot be sporadic and must arise from a functional counterpart between these $E$-functions.
\begin{theo}[Nesterenko-Shidlovskii, 1996]\label{thm:ns} With  the  notations of Theorem \ref{thm:sish}, 
there exists a finite set  $S$ (depending a priori on $Y(z)$) such that  for any $ \alpha\in \Qbar\setminus S$,  the following holds. For any homogeneous polynomial $P\in \Qbar[X_1,\ldots, X_n]$ such that $P(F_1( \alpha),\ldots, F_n( \alpha))=0$, there exists a polynomial $Q\in \Qbar[Z, X_1, \ldots, X_n]$, homogeneous in the variables $X_1,\ldots,X_n$, such that $Q(\alpha, X_1,\ldots, X_n)=P(X_1,\ldots, X_n)$ and
$
Q(z,F_1(z),\ldots, F_n(z))=0.
$
\end{theo}
The indetermination of the set $S$ is a problem.  It was lifted by Beukers~\cite{Beukers} using Andr\'e's theory of $E$-operators~\cite{andre}. 
\begin{theo}[Beukers, 2006]\label{theo:beukers} With   the  notations of Theorems \ref{thm:sish} and \ref{thm:ns}, one may choose  $S=\{\alpha \in \Qbar : \alpha T(\alpha)= 0\}$.
\end{theo}
A natural question is whether it is possible to determine algorithmically the (in)existence of a polynomial relation between values of given $E$-functions at algebraic points, and between these $E$-functions themselves. A difficult point is that we may be interested in $E$-functions $F_1,\ldots,F_p$ for which the vector $Y(z)=\tra(F_1(z), \ldots, F_p(z))$ {\em is not }  a solution of any differential system of the form $Y'(z)=A(z)Y(z)$.  Our main result answers this question.
\begin{theo} \label{th1} There exists an algorithm to perform the following three tasks. Given as inputs an integer $p\ge 1$ and some $E$-functions $F_1(z)$, \ldots, $F_p(z)$, 
\begin{enumerate}
\item[$(i)$] it computes a system of generators of the ideal of polynomial relations between $F_1(z)$, \ldots, $F_p(z)$ with coefficients in $\Qbar(z)$; 
\item[$(ii)$] given any $\alp\in\Qbar$, it computes a system of generators of the ideal of polynomial relations between the values  $F_1(\alp)$, \ldots, $F_p(\alp)$ with coefficients in $\Qbar$;
\item[$(iii)$] if $F_1(z)$, \ldots, $F_p(z)$ are algebraically independent over $\Qbar(z)$,  it determines the finite set of all $\alp\in\Qbar$ such that the values 
  $F_1(\alp)$, \ldots, $F_p(\alp)$ are algebraically dependent over $\Qbar$.
\end{enumerate}
\end{theo}
See \cite[\S 2.1]{AR} for an explanation of how an $E$-function is {\em given} by a differential equation with coefficients in $\Qbar(z)$ and sufficiently many Taylor coefficients to compute any of them from the differential equation.
A complex algebraic number  $\beta$ is {\em determined} or {\em computed} if we know  an explicit non-zero polynomial $P\in \mathbb Q[X]$ such that $P(\beta)=0$, together with a 
numerical approximation of $\beta$  sufficiently accurate to distinguish 
 $\beta$ from the other roots of~$P$.
 
In Theorem \ref{th1}, let us assume that the output of the algorithm in $(i)$ is that $F_1(z)$, \ldots, $F_p(z)$ are algebraically independent over $\Qbar(z)$. Though it is not an assumption of Theorem \ref{th1} that $\tra(F_1(z), \ldots, F_p(z))$ be a solution of a differential system $Y'(z)=A(z)Y(z)$ with $A(z)\in M_p(\Qbar(z))$, let us further assume that this is the case. Then, by Theorem \ref{theo:beukers}, the finite set of  algebraic numbers in $(iii)$ is a subset of $\{\alpha \in \Qbar : \alpha T(\alpha)=0\}$; it contains $0$ since $F_1(0)$, \ldots, $F_p(0)$ are algebraic numbers. In other words, the algorithm determines in $(iii)$ which roots $\xi$ of $T$ provide a polynomial relation between $F_1(\xi)$, \ldots, $F_p(\xi)$ with coefficients in $\Qbar$, and then for each $\xi$,  $(ii)$  describes all such relations. The problem in the general setting of Theorem \ref{th1}  is that no finite set $S$ containing the values $\alpha$ of $(iii)$ is known in advance. The most difficult part of the proof of  Theorem \ref{th1} is to construct such a finite set. Moreover, the putative algebraic independence of the functions $F_j$'s can be proven by various {\em ad hoc} means, and not necessarily by the complicated algorithm in $(i)$. The latter is a variation of the Hrushovski-Feng algorithm which, to the best of our knowledge, has not yet been implemented in any computer algebra system.

The case $p=1$ of Theorem \ref{th1} has been proved in \cite{AR}:
\begin{theo}[Adamczewski-R., 2018] \label{theo:AR}
There exists an algorithm to perform the following tasks. Given $F(z)$ an $E$-function as input, it first says whether $F(z)$ is transcendental over $\Qbar(z)$ or not. If it is transcendental, it then outputs the finite list of algebraic numbers $\alpha$ such that $F(\alpha)$ is algebraic,
together with the corresponding list of values $F(\alpha)$.
\end{theo}

The proof of Theorem \ref{th1} shares certain characteristics with that of Theorem \ref{theo:AR}, in particular Beukers' lifting results in \cite{Beukers} will form our starting point concerning $E$-functions, and we shall also need to compute the  minimal non-zero differential equation satisfied by an $E$-function. But our proof is not an adaptation, as we need new ideas. Indeed, we shall make an important use of methods coming from commutative algebra (in particular \GB bases), which {\em a contrario} were not used in \cite{AR}. In particular, 
 in the case $p=1$  Theorem \ref{th1} provides an algorithm different from the one of Theorem \ref{theo:AR}. 
 
 \bigskip
 
 The paper is organized as follows. In \S\ref{secGB}, we recall certain standard facts in elimination theory we shall need in the proof of Theorem \ref{th1} in the subsequent sections. In \S\ref{subsecpol}, we explain part $(i)$, which is not really new and the proof of which is given for the reader's convenience. In \S\ref{sec:linindepE}, 
we show that we can assume in $(ii)$ and $(iii)$ that $F_1(z), \ldots, F_p(z)$ are linearly independent over $\Qbar(z)$. In \S\ref{sec:reduc23commalg}, we then reduce parts $(ii)$ and $(iii)$ to  a problem of commutative algebra using Beukers' lifting results. We solve this problem (which may be of independent interest) in    \S \ref{sec6} by modifying Buchberger's algorithm. At last,   \S \ref{sec7} is devoted to examples. 

 \bigskip

We conclude with the following remark.  In his original paper, Siegel gave a slightly more general definition of $E$-functions: the upper bounds 
$\vert \sigma(a_n)\vert \leq C^{n+1}$
 and $1\le  d_n  \leq D^{n+1}$ in $(ii)$ and $(iii)$ were replaced with $\vert \sigma(a_n)\vert \leq n!^{\varepsilon}$ and
 $1\le  d_n  \leq   n!^{\varepsilon}$ for any $\varepsilon>0$, provided $n$ is large enough with respect to  $\varepsilon$.
 Theorems \ref{thm:sish} and \ref{thm:ns} hold in this more general setting. Beukers' proof of Theorem \ref{theo:beukers} does not, but Andr\'e has proved \cite{andreens} a general result, valid for $E$-functions in Siegel's sense, which contains Theorem \ref{theo:beukers}. In the present paper we shall use also an effective result due to Beukers  \cite[Theorem 1.5]{Beukers} to get rid of non-zero singularities, which has been recently generalized to $E$-functions in Siegel's sense by Lepetit \cite{lepetit}. Therefore all results we prove in this paper are valid in this setting, and the same remark applies to Theorem \ref{theo:AR} proved in  \cite{AR}.

\section{\GB bases and elimination: standard facts} \label{secGB}

In this section, we recall standard 
facts about \GB bases and elimination. We refer to any textbook on this topic (for instance \cite{BW, AECF, Cox}) for details and proofs. 

\medskip

Let $\L$ be a field, and $I $ be an ideal of the polynomial ring $\L[T_1,\ldots,T_N]$. Let $i \in \{1,\ldots,N\}$ be an integer, fixed throughout this section: in Proposition \ref{propelimGB} below we shall  compute a system of generators of the intersection  $I\cap\L[T_1,\ldots,T_i]$. 

\medskip

A {\em monomial} is an element of $\L[T_1,\ldots,T_N]$ of the form $\Tsoul^\asoul = T_1^{a_1}\ldots T_N^{a_N}$ with $\asoul = (a_1,\ldots,a_N)\in\N^N$. Given $\asoul, \bsoul\in\N^N$ we say that $\Tsoul^\asoul$ is less than $\Tsoul^\bsoul$, and we write $\Tsoul^\asoul < \Tsoul^\bsoul$, if either $\sum_{j=1}^i a_j < \sum_{j=1}^i b_j$ or: ( $\sum_{j=1}^i a_j = \sum_{j=1}^i b_j$ and $\asoul$ is less than $\bsoul$ in the lexicographical order on $\N^N$). This specific order, called the {\em $i$-th elimination order}, is useful to us because our purpose is to study $I\cap\L[T_1,\ldots,T_i]$ (see Proposition \ref{propelimGB} below). A monomial $\Tsoul^\asoul$ is  said to be {\em divisible} by  $\Tsoul^\bsoul$ if $c_j = a_j-b_j$ is non-negative for any $j \in \{1,\ldots,N\}$; then we write $\frac{\Tsoul^\asoul}{\Tsoul^\bsoul} = \Tsoul^\csoul$. 

Any non-zero polynomial $P\in \L[T_1,\ldots,T_N]$ can be written in a unique way as a  linear combination $\lambda_1  \Tsoul^{\asoul_1} + \ldots +  \lambda_r  \Tsoul^{\asoul_r}$ with non-zero coefficients $\lambda_1,\ldots,\lambda_r\in\L$ of  decreasing monomials $  \Tsoul^{\asoul_1}  > \ldots > \Tsoul^{\asoul_r}$. Then $\Tsoul^{\asoul_1}$ is called the {\em leading monomial} of $P$, and we write $\Tsoul^{\asoul_1} = \lmon{P}$. In the same way, $\asoul_1 = \lexp{P}$ is the {\em leading exponent}, $\lambda_1 = \lcoeff{P}$ is the {\em leading coefficient}, and $\lambda_1 \Tsoul^{\asoul_1} = \lterm{P}$ is the {\em leading term} of $P$. 

A \GB basis (or standard basis) of $I$, with respect to the order $<$ we have chosen, is a family $(P_1,\ldots,P_r)$ of non-zero elements of $I$ with the following property: for any $P\in I \setminus\{0\}$ there exists $k\in\{1,\ldots,r\}$ such that $\lmon{P}$ is divisible by $\lmon{P_k}$. An important property is that any \GB basis generates the ideal $I$. However no minimality property is assumed: adding arbitrary elements of $I \setminus\{0\}$ to a \GB basis always provides a  \GB basis. Starting with a system of generators of $I$, a usual way to construct a \GB basis is Buchberger's algorithm (i.e., lines 1 and 2 of Algorithm 1 presented below). To state it we need some more notation.

Given any family  $P,P_1,\ldots,P_r \in \L[T_1,\ldots,T_N]$ of non-zero polynomials, we consider the following operation. Choose (if possible) an index $k\in\{1,\ldots,r\}$ such that   $\lmon{P}$ is divisible by $\lmon{P_k}$, and replace $P$ with $P - \frac{\lterm{P}}{\lterm{P_k}} P_k$. After repeating this operation as many times as possible, $P$ is replaced with a polynomial $\widetilde P$ such that there is no   index $k$ for which  $\lmon{P}$ is divisible by $\lmon{P_k}$; possibly $\widetilde P=0$. This polynomial $\widetilde P$ is called a {\em remainder} in the weak division of $P$ by $P_1,\ldots,P_r$. Note that different choices of $k$ at some steps may lead to different remainders.

Given non-zero polynomials $P,Q \in \L[T_1,\ldots,T_N]$, their {\em $S$-polynomial} or {\em syzygy polynomial} is defined by
$$S(P,Q) = \frac{\lterm{Q}P - \lterm{P}Q}{\gcd(\lmon{P},\lmon{Q})}$$
where $\gcd( \Tsoul^\asoul , \Tsoul^\bsoul) =  \Tsoul^\csoul$ with $c_j = \min(a_j,b_j)$ for any $j\in \unN$. 

\medskip

We can now state the algorithm we are interested in in this section.

\bigskip

\fbox{\begin{minipage}{0.9\textwidth}
Input: a generating system $F$ of an ideal $I$ of $\L[T_1, \ldots, T_N]$.

Output: a \GB basis $H$ of $I\cap\L[T_1, \ldots, T_i]$.

1. $G \recoit F$

2. Repeat:

\quad \quad $a.$ $G' \recoit G$

\quad \quad $b.$ For $P,Q\in G'$ with $P\neq Q$ do:

\quad \quad \quad \quad $(i).$ Compute a remainder $R$ in the weak division of $S(P,Q)$ by $G'$

\quad \quad  \quad \quad $(ii).$ If $R\neq 0$:
 
\quad \quad \quad \quad    \quad \quad \quad $G \recoit G\cup\{R\}$

Until $G=G'$

3. $H \recoit G \cap \L[T_1, \ldots, T_i]$
\end{minipage}}

\medskip

\begin{center}
Algorithm 1: Computation of a \GB basis of $I\cap\L[T_1, \ldots, T_i]$.
\end{center}

\medskip

Lines 1 and 2 of Algorithm 1 are known as Buchberger's algorithm: the output $G$ is a \GB basis of $I$. Line 3 means that $H$ contains those polynomials in $G$ which depend only on the variables $T_1$, \ldots, $T_i$. Then $H$ is a \GB basis of $I\cap\L[T_1, \ldots, T_i]$  by the following result (see \cite[Exercise 24.4]{AECF}), which concludes the proof that Algorithm 1 works as announced. 

\begin{prop} \label{propelimGB} Let $P_1,\ldots,P_r$ be a \GB basis of $I$ with respect to the $i$-th elimination order. Then those $P_j$ which 
depend  only on the variables $T_1$, \ldots, $T_i$ make up a  \GB basis of $I\cap\L[T_1, \ldots, T_i]$.
\end{prop}

\medskip

We conclude this section with basic facts about polynomials in $T_1$,  \ldots, $T_N$ that depend on an auxiliary parameter $z$; this will be the setting of the proof of part $(i)$ of Proposition~\ref{proputile} in \S \ref{sec6} below.

Let $\K$ be a  subfield of $\C$, and $\L = \K(z)$. We fix $W(z)\in\K[z]\setminus\{0\}$ and consider polynomials $P\in\K[z,\frac1{W(z)},T_1,\ldots,T_N]  \subset \L[T_1,\ldots,T_N] $. We also fix $\alp\in\K$ such that $W(\alp)\neq0$. Then any such $P$ can be evaluated at $z=\alp$; we denote by $P_\alp \in\K[T_1,\ldots,T_N]$ the polynomial obtained in this way.

An easy (but already instructive) example is the following: $P = (z-1)T_1+T_2 $. Assume that $i\geq 2$ and consider as above the $i$-th elimination order, so that $T_1 > T_2$. Then in $\L[T_1,\ldots,T_N]$ we have $\lmon{P} = T_1$. However the leading monomial of $P_\alp$ depends on $\alp$: it is $T_1$ if $\alp\neq 1$, but $T_2$ if $\alp=1$. In general, $\lmon{P_\alp}$ can be easily determined for almost all values of $\alp$:
$$ 
\lmon{P_\alp} =   \lmon{P} \quad \mbox{ if }\;(\lcoeff{P})(\alp)\neq 0.$$
In particular, $ (\lcoeff{P})(\alp)\neq 0$ implies $P_\alp\neq 0$. 

In the same way  we have 
$$S(P,Q)_\alp = S(P_\alp,Q_\alp)  \quad \mbox{ if } (\lcoeff{P})(\alp)\neq 0 \; \mbox{ and }\; (\lcoeff{Q})(\alp)\neq 0.$$

\section{Algebraic relations between $F_1(z)$, \ldots, $F_p(z)$} \label{subsecpol}

In this section, we briefly explain  the proof of part $(i)$ in Theorem \ref{th1}. It is a modification of one of the steps in Feng's algorithm \cite{Feng} to compute differential Galois groups, which is itself based on Hrushovski's algorithm \cite{hrus}.

For any $1\leq i \leq p$, we are given a differential equation of order $n_i$ satisfied by $F_i$. Then the vector $Y$ with $n = n_1+\ldots+n_p$ coordinates $F_i^{(j)}$, where $1\leq i \leq p$ and $0\leq j \leq n_i-1$, is a solution of a differential system $Y'=AY$ with $A\in M_n(\Qbar(z))$. Let $Y_1=Y$, $Y_2$, \ldots, $Y_n$ be a basis of solutions of this differential system; in other words, the matrix with columns $Y_1$, \ldots, $Y_n$ is a fundamental matrix of solutions. As pointed out to us by Feng, Algorithm~4.1 of \cite{Feng} in which Step $(h)$ is replaced with \cite[Proposition 3.6]{Feng} provides an algorithm to compute a system of generators of the ideal of algebraic relations over $\Qbar(z)$ among the coordinates of $Y_1$, \ldots, $Y_n$  (i.e., among the coefficients of this fundamental matrix of solutions). Using \GB bases (see Algorithm 1 in \S \ref{secGB}), we then deduce a system of generators of the ideal of algebraic relations over $\Qbar(z)$ among $F_1$, \ldots, $F_p$, since they are amongst the coordinates of~$Y_1$.

This concludes the proof of part $(i)$ of Theorem \ref{th1}.

\section{Linear dependence relations between $E$-functions} \label{sec:linindepE}

In this section, we prove that  in parts $(ii)$ and $(iii)$ of Theorem  \ref{th1}, we may assume that $F_1(z)$, \ldots, $F_p(z)$ are $\Qbar(z)$-linearly independent.
Let us start with a lemma, of independent interest, in the statement of which it is not necessary to assume that the functions $y_j$ are $E$-functions.

\begin{lem} \label{lemlindep}
Let $Y = \tra (y_1,\ldots,y_n)$ be a solution of a differential system $Y'=AY$ with $A\in M_n(\Qbar(z))$. Let 
$$\calR_Y = \{ (P_1,\ldots,P_n)\in \Qbar(z)^n, \, \, P_1(z) y_1(z) + \ldots + P_n (z) y_n(z) = 0\}$$
be the $\Qbar(z)$-vector space of $\Qbar(z)$-linear relations between $y_1(z)$, \ldots, $y_n(z)$. Then there is an algorithm to compute a basis of this vector space; accordingly it enables one to know whether $y_1$, \ldots, $y_n$ are linearly independent over $\Qbar(z)$ or not.
\end{lem}

\begin{proof}The cyclic vector theorem provides an invertible matrix $P\in {\rm GL}_n(\Qbar(z))$ (which depends only on $A$) such that $V = PY$ satisfies $V'=CV$, where $C\in M_n(\Qbar(z))$ is a companion matrix. In other words, letting $V = \tra (v_1(z),\ldots, v_n(z))$ we have $v_i = v_1^{(i-1)}$ for any $1\leq i \leq n$, and $Lv_1 = 0$ for some differential operator $L\in\Qbar(z) [\frac{d}{dz}]$ of order $n$. Moreover $P$ and $L$ can be computed effectively (see for instance \cite[Chapter 2, \S 2.1]{VDPS}). 

Let $L_0\neq 0$ denote a differential operator in $\Qbar(z) [\frac{d}{dz}]$ such that $L_0 v_1=0$, of minimal order $k$. Then \cite[Theorems 1-2]{BRS} provides an explicit integer $D$ for which such an $L_0$ exists of the form $\sum_{i=0}^k Q_i(z) (\frac{d}{dz})^i$ with $Q_i\in \Qbar[z]$ of degree at most $D$. Using the multiplicity estimate  of \cite[Th\'eor\`eme 1]{BB} completed by \cite[Lemma 3.1]{BCY}, an algorithm to compute explicitly such an $L_0$ is given in \cite[\S 3]{AR}, and we refer to the discussion surrounding \cite[Theorem~3]{BRS} for further details about it.

Then $v_1$, $v_1'$, \ldots, $v_1^{(k-1)}$ are linearly independent over $\Qbar(z)$, and we have 
$$v_1^{(k)}(z) = - \sum_{i=0}^{k-1} \frac{Q_i(z)}{Q_k(z)} v_1^{(i)}(z).$$
Taking successive derivatives of this relation we can express each $v_j = v_1^{(j-1)}$, with $k+1\leq j \leq n$, as a $\Qbar(z)$-linear combination of $v_1$, $v_1'$, \ldots, $v_1^{(k-1)}$. Since $Y = P^{-1}V$ we deduce an explicit expression of $y_1$, \ldots, $y_n$ as $\Qbar(z)$-linear combinations of the $\Qbar(z)$-linearly independent  functions  $v_1$, $v_1'$, \ldots, $v_1^{(k-1)}$.  Therefore Lemma  \ref{lemlindep} boils down to the problem of finding linear relations between the columns of a matrix: it can be easily solved using Gaussian elimination.\end{proof}

\medskip

We now apply Lemma \ref{lemlindep} to prove that in parts $(ii)$ and $(iii)$ of Theorem  \ref{th1} we may assume that $F_1(z)$, \ldots, $F_p(z)$ are $\Qbar(z)$-linearly independent.

Let $F_1(z)$, \ldots, $F_p(z)$ be $E$-functions. Using Lemma \ref{lemlindep} we may compute a maximal subset $F_{i_1}(z)$, \ldots, $F_{i_t}(z)$ of $\Qbar(z)$-linearly independent functions among $F_1(z)$, \ldots, $F_p(z)$, and an expression
\begin{equation} \label{eqdeplinramener}
F_i(z) = \sum_{j=1}^t Q_{i,j} (z) F_{i_j}(z) \mbox{ for each } i \in \{1,\ldots,p\} \setminus \{i_1,\ldots,i_t\}
\end{equation}
with $Q_{i,j}(z) \in \Qbar(z)$.

\medskip

If $t<p$ then part $(iii)$ of Theorem \ref{th1} is empty; to prove part $(ii)$ in this case, it is enough to compute a system of generators of the ideal of polynomial relations between $F_{i_1}$, \ldots, $F_{i_t}$ over $\Qbar(z)$. Indeed adding the relations \eqref{eqdeplinramener} to this system provides a system of generators of the ideal of polynomial relations between $F_1(z)$, \ldots, $F_p(z)$.

\medskip

Therefore, to complete the proof of Theorem \ref{th1}  what remains to do (taking \S \ref{subsecpol} into account) is to prove parts $(ii)$ and $(iii)$ under the additional assumption that $F_1(z)$, \ldots, $F_p(z)$ are linearly independent over $\Qbar(z)$.

\section{Reduction of parts $(ii)$ and $(iii)$ of Theorem \ref{th1} to statements in commutative algebra} \label{sec:reduc23commalg}

The following proposition shows how to reduce parts $(ii)$ and $(iii)$ of Theorem \ref{th1} (in the case where $F_1(z)$, \ldots, $F_p(z)$ are  linearly independent over $\Qbar(z)$) to a problem in commutative algebra, that we shall solve in \S \ref{sec6}.  To begin with, we point out that $F_1(0)$, \ldots,  $F_p(0)$ are algebraic numbers so that $\alpha=0$ always belongs to the set of part $(iii)$, and part $(ii)$ is trivial for  $\alpha=0$. Therefore throughout this section, $\alpha$ denotes a non-zero algebraic number.

We denote by $\Xsoul$ the set of variables $X_1$, \ldots, $X_N$. In the following result, by ``compute an ideal $I$'' we mean ``compute a system of generators of $I$''.

\begin{prop} \label{propclef} Let $F_1(z)$, \ldots, $F_p(z)$ be $E$-functions  linearly independent over $\Qbar(z)$. 
There exists an algorithm to compute an integer $N\geq 1$, an ideal $I$ of $\Qbar[z, X_1,\ldots,X_N]$ and $\Qbar[z]$-linearly independent polynomials  $\varphi_1,\ldots,\varphi_p\in\Qbar[z][\Xsoul]$   homogeneous of degree 1 with respect to $X_1$,\ldots, $X_N$ with the following properties:
\begin{itemize}
\item[$(a)$] For any $R\in \Qbar[z,  Y_1,\ldots,Y_p] $ we have:
$$R(z, F_1(z),\ldots,F_p(z))=0 \quad \mbox{ if, and only if, } \quad 
R(z,  \varphi_1(z, \Xsoul), \ldots,  \varphi_p(z, \Xsoul) ) \in I.$$
\item[$(b)$] For any $S\in \Qbar[Y_1,\ldots,Y_p] $ and any $\alp\in\Qbar\etoile$ we have $S(F_1(\alp),\ldots,F_p(\alp))=0$ if, and only if, 
there exists $Q\in I$ such that 
$$ S( \varphi_1(\alp, \Xsoul), \ldots,  \varphi_p(\alp, \Xsoul) ) = Q(\alp, X_1,\ldots,X_N).$$
\end{itemize}
\end{prop}

In particular, $F_1$, \ldots, $F_p$ are algebraically  independent over $\Qbar(z)$  if, and only if,
$$I \cap \Qbar [ z,  \varphi_1(z, \Xsoul), \ldots,  \varphi_p(z, \Xsoul) ] = \{0\}.$$

\medskip

In this section we prove Proposition \ref{propclef} by applying two results of Beukers \cite{Beukers} on $E$-functions.

\medskip

\begin{proof}
Let $F_1$, \ldots, $F_p$ be $E$-functions  linearly independent over $\Qbar(z)$. Recall that each $F_i$ is given with a differential equation of order $n_i$ it satisfies. Let $\calF$ denote the $\Qbar(z)$-vector space generated by the $E$-functions $F_i^{(j)}$, $1\leq i \leq p$, $0\leq j \leq n_i-1$. Lemma \ref{lemlindep} enables us to compute the dimension of $\calF$, denoted by $N$. Since  $F_1$, \ldots, $F_p$ are  linearly independent over $\Qbar(z)$ we have $N\geq p$. Moreover Lemma \ref{lemlindep} shows also how to pick up $E$-functions $F_{p+1}$, \ldots, $F_N$ among the $F_i^{(j)}$ (with $j\geq 1$) such that $(F_1,\ldots,F_p, F_{p+1}, \ldots, F_N)$ is a basis of $\calF$ over $\Qbar(z)$.

Since $\calF$ is stable under derivation, each $F_i'$ (with $1\leq i \leq N$) is a linear combination of $F_1$, \ldots, $F_N$ with coefficients in $\Qbar(z)$: the vector $Y = \tra (F_1,\ldots,F_N)$ is a solution of a differential system $Y'=AY$ with $A\in M_N(\Qbar(z))$. Moreover the derivatives of $F_1$, \ldots, $F_N$ are explicit linear combinations of the $F_i^{(j)}$, and therefore of $F_1$, \ldots, $F_N$ using Lemma \ref{lemlindep}: the matrix $A$ is effectively computable.

\medskip

Our strategy is to apply Beukers' Theorem \ref{theo:beukers}. However $\alp$ might be a singularity of the differential system $Y'=AY$ (i.e., $T(\alpha)=0$ in the notation of Theorem \ref{theo:beukers}): to do this  we have to get rid of all non-zero singularities first. With this aim in view,  we apply \cite[Theorem 1.5]{Beukers} to the differential system $Y'=AY$ satisfied by the
vector $Y = \tra (F_1,\ldots,F_N)$ of which the coordinates are  $\Qbar(z)$-linearly independent $E$-functions. It provides $E$-functions $g_1$, \ldots,  $g_N$ and a matrix $M = (m_{j,k}(z))_{j,k}\in M_N(\Qbar[z])$ such that:
\begin{enumerate}
\item[$(i)$] For any $j\in\{1,\ldots,N\} $ we have $F_j(z) = \sum_{k=1}^N m_{j,k}(z) g_k(z)$.
 \item[$(ii)$] The vector $Z =  \tra(g_1,\ldots,g_N)$ is a solution of a differential system $Z'=BZ$ with $B\in M_N(\Qbar[z, 1/z])$.
 \end{enumerate}
The point here is that $0$ is the only possible finite singularity of the  differential system $Z'=BZ$. Moreover, the $E$-functions $g_1$, \ldots,  $g_N$, the polynomials $m_{j,k}(z)$ and the matrix $B$ are effectively computable (see \cite[\S 5]{AR}).

\medskip

Recall from $(i)$  that  $F_j(z) = \sum_{k=1}^N m_{j,k}(z) g_k(z)$ for any $j\in\{1,\ldots,N\} $; this relation is specially interesting for $j\leq p$, since $F_1$, \ldots, $F_p$ are the $E$-functions involved in the statement of Proposition \ref{propclef}. We let 
$$\varphi_j(z,X_1,\ldots,X_N) = \sum_{k=1}^N m_{j,k}(z) X_k \in \Qbar[z,X_1,\ldots,X_N]  \mbox{ for } 1\leq j \leq p,$$
so that 
\begin{equation} \label{eqfj}
F_j(z) = \varphi_j(z, g_1(z), \ldots, g_N(z))   \mbox{ for } 1\leq j \leq p.
\end{equation}
Then $\varphi_1$, \ldots, $\varphi_p$ are linearly independent over $\Qbar[z]$ because   $F_1$, \ldots, $F_p$ are.

\medskip

As mentioned in \S \ref{subsecpol} above, Feng's algorithm provides a system of generators of the ideal $I$ of polynomial relations between $g_1$, \ldots,  $g_N$:
$$I = \{ Q \in \Qbar[z,X_1,\ldots,X_N] \quad \mbox{ such that }\quad  Q(z,g_1(z),\ldots,g_N(z))=0\}.$$
Now for any $R\in \Qbar[z,  Y_1,\ldots,Y_p]$, Eq. \eqref{eqfj} yields:
$$R(z, F_1(z),\ldots,F_p(z))=  R(z,  \varphi_1(z, \gsoul(z)), \ldots,  \varphi_p(z, \gsoul(z)));$$
here and below we write $\gsoul(z)$ for the tuple $g_1(z),\ldots,g_N(z)$. Therefore $R(z, F_1(z),\ldots,F_p(z))$ is identically zero if and only if $R(z,  \varphi_1(z, \Xsoul), \ldots,  \varphi_p(z, \Xsoul) ) \in I$, thereby proving part $(a)$ of Proposition \ref{propclef}.

\medskip

To prove part $(b)$, let $\alp\in\Qbar\etoile$ and  $S\in \Qbar[Y_1,\ldots,Y_p] $. To begin with, assume that 
$$ S( \varphi_1(\alp, \Xsoul), \ldots,  \varphi_p(\alp, \Xsoul) ) = Q(\alp, X_1,\ldots,X_N) $$
for some $Q\in I$. Then we have:
\begin{eqnarray*}
S(F_1(\alp),\ldots,F_p(\alp))
&=& S(  \varphi_1(\alp, \gsoul(\alp)), \ldots,  \varphi_p(\alp, \gsoul(\alp))) \quad \mbox{ using Eq. \eqref{eqfj}}\\
&=& Q( \alp, g_1(\alp),\ldots,g_N(\alp)) \\
&=& 0 \quad \mbox{ since } Q\in I.
\end{eqnarray*}
Conversely, assume that $S(F_1(\alp),\ldots,F_p(\alp))=0$. Using Eq. \eqref{eqfj} we have 
$$S(  \varphi_1(\alp, \gsoul(\alp)), \ldots,  \varphi_p(\alp, \gsoul(\alp)))=0.$$
Consider the polynomial $P\in\Qbar[X_1,\ldots,X_N]$ defined by 
$$P(\Xsoul) = S(\varphi_1(\alp,\Xsoul), \ldots, \varphi_p(\alp,\Xsoul)) $$
so that we have
$$P(  g_1(\alp), \ldots, g_N(\alp))=0.$$
 Now   $\alp\neq 0$ is not a singularity of the differential system $Z' = BZ$ satisfied by $Z  =  \tra (g_1,\ldots,g_N)$. Therefore Beukers'  version  of the Siegel-Shidlovskii theorem (namely \cite[Theorem 1.3]{Beukers} or Theorem \ref{theo:beukers} above)  provides $Q \in \Qbar[z, X_1,\ldots,X_N]$ such that 
$$Q(z,g_1(z),\ldots,g_N(z))=0 \mbox{ and } Q(\alp, \Xsoul) =P( \Xsoul) = S(\varphi_1(\alp,\Xsoul), \ldots, \varphi_p(\alp,\Xsoul)) .$$
By definition of $I$ we have $Q\in I$. 
This concludes the proof of  Proposition \ref{propclef}.\end{proof}

\section{Completion of the proof of Theorem \ref{th1}: an algorithm in  commutative algebra} \label{sec6}

In this section, we complete the proof of Theorem \ref{th1}.

\medskip

Let $\K$ be a subfield of $\C$ on which arithmetic operations are implemented; it need not necessarily be $\Qbar$ or a number field at this stage.  We denote by $\Xsoul$ the set of variables $X_1$, \ldots, $X_N$. 

Let $\varphi_1,\ldots,\varphi_p\in\K[z,\Xsoul]$ be homogeneous of degree 1 with respect to $X_1$,\ldots, $X_N$ (i.e., linear forms in $X_1$, \ldots, $X_N$ with coefficients in $\K[z]$); assume that $\varphi_1,\ldots,\varphi_p$ are linearly independent over $\K[z]$. 

Let $I$ be an ideal of $\K[z,\Xsoul]$,  generated by $Q_1$,\ldots, $Q_\ell$. For any $\alp \in\K$, denote by
$J_\alp $ the set of all polynomials $S \in \K[Y_1,\ldots,Y_p] $ for which there exists $Q\in I$ with 
$$ S( \varphi_1(\alp, \Xsoul), \ldots,  \varphi_p(\alp, \Xsoul) ) =Q(\alp, \Xsoul).$$
 If $\K = \Qbar$ and $N$, $I$, $\varphi_1$, \ldots, $\varphi_p$ are provided by Proposition \ref{propclef}, then for any $\alpha\in\Qbar$
$$J_\alp = \{ S \in \Qbar[Y_1,\ldots,Y_p] , \, S(F_1(\alp),\ldots,F_p(\alp))=0\} $$
is the ideal considered in Theorem~\ref{th1}. Therefore combining Proposition~\ref{propclef} and Proposition~\ref{proputile} below concludes the proof of  Theorem~\ref{th1} (recall that the case $\alpha=0$ is trivial since $F_1(0)$, \ldots,  $F_p(0)$ are algebraic numbers). 

\medskip

In the following statement, both algorithms take $\varphi_1,\ldots,\varphi_p$,  $Q_1$,\ldots, $Q_\ell$  as inputs, and also  $\alp$ for $(ii)$.

\begin{prop} \label{proputile}
In this setting:
\begin{itemize}
\item[$(i)$]  If $ I  \cap \K[z, \varphi_1(z,\Xsoul),\ldots, \varphi_p(z,\Xsoul)] = \{0\}$, then there exists an algorithm to compute a non-zero polynomial $W\in\K[z]$ with the following property: for any $\alp\in\K$ such that $ J_\alp  \neq    \{0\}$, we have  $W(\alp) =  0$.
\item[$(ii)$] There exists an algorithm that, given $\alp\in\K$,  computes a  system of generators of the ideal $J_\alp $.  In particular it enables one to know whether $J_\alp$ is equal to $\{0\}$ or not.
\end{itemize}
\end{prop}
\noindent {\em Nota Bene}: In assertion $(ii)$, we do not need to assume that 
$$ 
I  \cap \K[z, \varphi_1(z,\Xsoul),\ldots, \varphi_p(z,\Xsoul)] = \{0\};
$$ 
this assumption is needed in $(i)$ to ensure that $W$ is non-zero. Moreover, after computing $W$  in $(i)$, it is possible to apply $(ii)$ to all roots of $W$: this allows one to determine exactly the (finite) set of all $\alp\in\K$ such that $J_\alp\neq\{0\}$.

The polynomial $W$ plays the role of the polynomial $u_0$ in \cite{AR}.

\bigskip

In the rest of this section we shall prove Proposition \ref{proputile}.

\medskip

 We denote by  $I_\alp$  the ideal of $\K[\Xsoul]$ consisting in all polynomials $Q(\alp,\Xsoul)$ with $Q\in I$, and   by  $\chialp$   the linear map $\K[Y_1,\ldots,Y_p] \to \K[X_1,\ldots,X_N]$ defined by
$$\chialp(S(Y_1,\ldots,Y_p))  =  S( \varphi_1(\alp, \Xsoul), \ldots,  \varphi_p(\alp, \Xsoul) ).$$
Then we have $J_\alp = \chialp^{-1} (I_\alp)$.

\medskip

\begin{proof}[Proof of part $(ii)$ of Proposition \ref{proputile}]  
 Fix $\alp\in\K$. Denote by $r$ the dimension of the $\K$-vector space spanned by the linear forms $\varphi_j(\alp,\Xsoul)$ with $1\leq j \leq p$; we have $0\leq r \leq \min(p,N)$. There exist effectively computable indices $1\leq j_1 < \ldots < j_r\leq p$ and $1\leq i_1 < \ldots < i_{N-r} \leq N$ such that $\varphi_{j_
1}(\alp,\Xsoul)$, \ldots,  $\varphi_{j_r}(\alp,\Xsoul)$, $X_{i_1}$, \ldots,  $X_{i_{N-r}}$ is a basis of the $N$-dimensional vector space of $\K$-linear combinations of $X_1$,\ldots, $X_N$. In general there are several such tuples $(i_1,\ldots,i_{N-r},j_1,\ldots,j_r)$; we choose (arbitrarily) the least in lexicographical order. We let $T_1 = \varphi_{j_1}(\alp,\Xsoul)$, \ldots,  $T_r= \varphi_{j_r}(\alp,\Xsoul)$, $T_{r+1} = X_{i_1}$, \ldots,  $T_N = X_{i_{N-r}}$. In this way $T_1,\ldots,T_N$ are linearly independent linear forms in $X_1,\ldots,X_N$, and are therefore algebraically independent. We have $\K[\Xsoul] = \K[\Tsoul]$ where $\Tsoul$ stands for $T_1,\ldots,T_N$, and any polynomial in $\K[\Xsoul]$ can be written in a unique way as a polynomial in $T_1,\ldots,T_N$ with coefficients in $\K$. Algorithm 1 described in \S \ref{secGB}, with $\L=\K$ and $i=r$, enables one (starting with $Q_1(\alp,\Xsoul)$, \ldots,  $Q_\ell(\alp,\Xsoul)$) to compute a \GB basis of $I_\alp \cap \K[T_1,\ldots,T_r]  = I_\alp\cap\Im(\chialp)$. Each element of this \GB basis is of the form $P(T_1,\ldots,T_r)$, and we have
$$\chialp(P(Y_{j_1},\ldots,Y_{j_r})) = P(T_1,\ldots,T_r).$$
Let $\calB_1$ be the set of all polynomials $P(Y_{j_1},\ldots,Y_{j_r})$ for $P(T_1,\ldots,T_r)$ in this \GB basis; then $\chialp(\calB_1)$ is a set of generators of $I_\alp\cap \Im(\chialp)$. On the other hand, for each $j\in \{1,\ldots,p\}\setminus\{j_1,\ldots,j_r\}$ there exist scalars $\lambda_{j,t}\in\K$ (for $1\leq t \leq r$) such that $\varphi_{j}(\alp,\Xsoul)= \sum_{t=1}^r \lambda_{j,t} \varphi_{j_t} (\alp,\Xsoul)$; we let  $\calB_2$ be the set of all linear polynomials  $Y_j- \sum_{t=1}^r \lambda_{j,t} Y_{j_t} $ for $j\in \{1,\ldots,p\}\setminus\{j_1,\ldots,j_r\}$. Then $\calB_2$ is a set of generators of the ideal $\ker(\chialp)$, so that 
$\calB_1\cup\calB_2$ is a set of generators of the ideal $ J_\alp = \chialp^{-1} ( I_\alp )$. This set is empty if, and only if, $J_\alp=\{0\}$. This concludes the proof of part $(ii)$ of  Proposition \ref{proputile}.\end{proof}   

\medskip

\begin{proof}[Proof of part $(i)$ of Proposition \ref{proputile}]  
  We first point out that the algorithm   described above for part $(ii)$  depends on $\alp$ in many ways, through $r$, $i_1$, \ldots, $i_{N-r}$, $ j_1$, \ldots, $j_r$, and at each step of Algorithm~1 (whenever a remainder or a syzygy polynomial is computed, or the equality of two polynomials is tested).
We refer to the end of \S \ref{secGB} for examples where $\lmon{P(\alp,\Xsoul)}$ or $S( P(\alp,\Xsoul), Q(\alp,\Xsoul))$ depend on $\alp\in\K$. The general idea when several polynomials $P\in\K(z)[\Xsoul]$   are involved is that if none of their leading coefficients  vanishes at $\alp$, then everything goes smoothly: the leading monomial of each  $P(\alp, \Xsoul)$ is independent from $\alp$. Since the algorithm involves finitely many steps, 
 only finitely many polynomials are computed: the strategy is to compute a common multiple $W\in\K[z]\setminus\{0\}$ of the numerators of the leading coefficients of all polynomials $P\in\K(z)[T_1,\ldots,T_N]$ that appear during the algorithm. Then for any $\alp\in\K$ such that   
$W(\alp)\neq 0$,  the algorithm described above for part $(ii)$  takes place exactly in the same way, independently of $\alp$: actually it follows exactly the same steps as if it were worked out over  $\K(z)$. We refer to the end of \S \ref{sec7} for an example.

To make  this strategy more precise, let $M_0(z)=(m_{j,k}(z)) \in M_{p,N}(\K[z])$ denote the matrix defined by $\varphi_j(z,\Xsoul) = \sum_{k=1}^N m_{j,k}(z) X_k$ for any $j\in\{1,\ldots,p\}$. Since $\varphi_1$, \ldots, $\varphi_p$ are linearly independent over $\K[z]$, the matrix $M_0(z)$ has rank $p$: there exists a minor $W_0(z)$ of $M_0(z)$, of size $p$, which is not identically zero. If $\alp\in\K$ is such that $W_0(\alp)\neq 0$, then $M_0(\alp)$ has rank $p$ and the integer $r$ defined in the proof of part $(ii)$  (in terms of $\alp$) is equal to $p$. 

Let $ \iti_1 $,   \ldots,  $ \iti_{N-p}$ denote the indices of the columns of $M_0(z)$ which do not appear in the submatrix of which $W_0(z)$ is the determinant, with 
$ 1 \leq \iti_1 < \ldots < \iti_{N-p} \leq N$. Choosing $W_0(z)$ properly among all non-zero minors of $M_0(z)$ of size $p$, we may assume that $(\iti_1  ,   \ldots  ,  \iti_{N-p})$ is least possible with respect to lexicographic order (i.e., all tuples less than $(\iti_1  ,   \ldots  ,  \iti_{N-p})$ correspond to zero minors). Then for any $\alp\in\K$ such that $W_0(\alp)\neq 0$, we have $i_1 =\iti_1$, \ldots, $i_{N-r} = \iti_{N-p}$ where $i_1$, \ldots, $i_{N-r}$ have been constructed in terms of $\alp$ in the proof of part $(ii)$  (recall also  the equality $r=p$, already noticed). Moreover, by definition we have $j_1 = 1$, \ldots, $j_r = p$ for such an $\alp$.

As in the proof of part $(ii)$,  we let $T_1 = \varphi_1(z,\Xsoul)$, \ldots,  $T_p= \varphi_p(z,\Xsoul)$, $T_{p+1} = X_{\iti_1}$, \ldots,  $T_N = X_{\iti_{N-p  }}$. In this way, $ T_1$, \ldots, $T_N $ make up  a basis of the $\K(z)$-vector space generated by  $X_1,\ldots,X_N$, and are therefore algebraically independent over $\K(z)$; we have $\K(z)[\Xsoul] = \K(z)[\Tsoul]$ where $\Tsoul$ stands for $T_1,\ldots,T_N$. 
By definition of $W_0(z)$ and $\iti_1$, \ldots, $\iti_{N-p}$, each $X_i$ with $1\leq i \leq N$   can be written as $\sum_{k=1}^N \lambda_{i,k}(z) T_k$ with $W_0(z)  \lambda_{i,k}(z) \in \K[z]$. In particular we have $\K[z,\Xsoul] \subset \K[z,\frac1{W_0(z)},\Tsoul]$. 

Now we   run  Algorithm 2 below, in which all multivariate polynomials are seen in $\K(z)[\Tsoul]$; we use the notation of \S \ref{secGB} with $\L = \K(z)$ and $i=p$. The input involves the set $F$   of polynomials $Q_k(z,\Xsoul)$ with $1\leq k \leq \ell$ which generate $I$; they belong to $\K[z,\Xsoul] \subset \K[z,\frac1{W_0(z)},\Tsoul] \subset \K(z)[\Tsoul]$ but not necessarily to $\K[z][\Tsoul]$.  The leading coefficient of any non-zero  $P \in \K(z)[\Tsoul]$ is a non-zero rational function $R(z) = N(z)/D(z)$ with $N,D\in\K[z]\setminus\{0\}$, $\gcd(N,D)=1$ and $D$ monic. Its  numerator $N(z)$ is denoted by $\numlcoeff{P}$, and more generally we define $\num{R}$ in this way for any non-zero $R\in\K(z)$.

\medskip

Except for line 5 and the computation of $W(z)$, Algorithm 2 below follows exactly Buchberger's algorithm (i.e., Algorithm~1 over $\L = \K(z)$ without the last step): at the end, $G$ is a \GB basis of the ideal $\Iti$ of $\K(z)[\Tsoul]$ generated by the input $F$. In particular it terminates (we shall prove later that line $5a$ can be carried out).

\bigskip

\fbox{\begin{minipage}{0.9\textwidth}
Input: Integers $1\leq  p \leq N$, a non-zero polynomial $W_0\in\K[z]$ as above, and 
 a finite subset  $F$ of $\K[z,\frac1{W_0(z)},\Tsoul]\setminus\{0\}$.

Output: a non-zero polynomial $W(z)\in\K[z]$ as in part $(i)$ of Proposition \ref{proputile}.

1. $G \recoit F$

2. $W \recoit W_0$

3. For $P\in G$ do:

 \quad \quad \quad $W \recoit \lcm( W, \numlcoeff{P})$

4. Repeat:

\quad \quad $a.$ $G' \recoit G$

\quad \quad $b.$ For $P,Q\in G'$ with $P\neq Q$ do:

\quad \quad \quad \quad \quad $(i).$  $W \recoit \lcm( W, \numlcoeff{P-Q})$

\quad \quad  \quad\quad  \quad $(ii).$ $S \recoit S(P,Q)$
 
\quad \quad \quad   \quad \quad $(iii).$  If $S\neq 0$:
  
\quad \quad \quad   \quad \quad   \quad \quad \quad \quad $W \recoit \lcm( W, \numlcoeff{S})$

\quad \quad \quad  \quad \quad $(iv).$ While $S\neq 0$ and  there exists $P_1\in G'$ such that $\lmon{P_1} $ divides $\lmon{S}$:
 
\quad \quad \quad \quad     \quad\quad  \quad \quad $\kappa.$ $ S \recoit S - \frac{\lterm{S}}{\lterm{P_1}}P_1$
    
\quad \quad \quad \quad \quad    \quad \quad \quad $\eta.$ If $S\neq 0$:

\quad \quad \quad \quad \quad \quad \quad    \quad \quad    \quad \quad  $W \recoit \lcm( W, \numlcoeff{S})$
  
\quad \quad \quad  \quad \quad $(v).$ If $S\neq 0$: 
 
\quad \quad \quad \quad \quad    \quad \quad \quad $G \recoit G\cup\{S\}$

Until $G=G'$

5. For   $P\in G$ do:

\quad \quad $a.$ Find $\asoul\in\N^N$ such that $a_i\geq 1$ for at least one integer $i\geq p +1$ and the coefficient $\lambda_\asoul(z)$ of $\Tsoul^\asoul$ in $P(z,\Tsoul)$ is non-zero.

\quad \quad $b.$  $W \recoit \lcm( W, \num{\lambda_\asoul(z)})$

\end{minipage}}

\medskip

\begin{center}
Algorithm 2: Computation of $W(z)$ in part $(i)$.
\end{center}

\medskip

At the end of Algorithm 2, $W$ is a non-zero polynomial that we denote by $\Wend$. We shall now prove that $\Wend$ satisfies the property $(i)$ of Proposition \ref{proputile}.

\medskip

Notice that $\Wend$ is constructed by taking least common multiples repeatedly, so that at each step of the algorithm $W$ divides $\Wend$. We claim that throughout the algorithm, 
\begin{equation} \label{eq:claim1}
P \in \K\left[z,\frac1{\Wend(z)},\Tsoul\right] \quad \mbox{ and } \quad \numlcoeff{P} \mbox{ divides  $ \Wend$ for any $P\in G$.}
\end{equation}
 This is true at line 1 using line 3 and the assumption $F \subset  \K[z,\frac1{W_0(z)},\Tsoul]$ where $W_0$ divides $\Wend$ using line 2. Whenever a new element $S$ is added to $G$ on line $4b(v)$, it is constructed in lines $4b(ii)$ and $4b(iv)\kappa$ in such a way that $S\in   \K[z,\frac1{\Wend(z)},\Tsoul]$  (since on line $4b(iv)\kappa$, $P_1$ has been inserted in $G$ previously so that $\numlcoeff{P_1}$ divides $\Wend$), and $\numlcoeff{S} $ divides $\Wend$ using line $4b(iv)\eta$. This proves the claim.
 From now on, we fix   $\alp\in\K$  such that $\Wend(\alp)\neq 0$. Claim \eqref{eq:claim1} shows that at any step of the algorithm,
$$ \mbox{ for any } P(z,\Tsoul)\in G, \, \,  P(\alp, \Tsoul)  \mbox{ exists and  }  (\lcoeff{P})(\alp)\neq 0.$$ 
For any $Q = Q(z,\Tsoul)\in   \K[z,\frac1{\Wend(z)},\Tsoul] $, denote by $Q_\alp =Q(\alp, \Tsoul)\in\K[\Tsoul]$ the polynomial obtained by evaluating at $z=\alp$ (recall  that  $\Wend(\alp)\neq0$, so that $\alp$ is not a pole of any coefficient of $Q$).  Then at each step of the algorithm, for any $P\in G$, $P_\alp$ exists and we have $\lmon{P_{\alp}}=\lmon{P}$ since $(\lcoeff{P})(\alp)\neq 0$.
 In the same way, at each step, for any $P,Q\in G'$ with $P\neq Q$ we have $\lmon{P_\alp-Q_\alp} = \lmon{P-Q}$ so that $P_\alp\neq Q_\alp$ (using line $4b( i)$ to check that $\lmon{P-Q}(\alp)\neq 0$). Lines $4b (iii)$ and $4b(iv)\eta$ show that $\lmon{\widetilde{R}}=\lmon{R}$, where $R$ is the remainder in the weak division of $S(P,Q)$ by $G'$ computed over $\K(z)$, and $\widetilde{R}$   is  the remainder in the weak division of $S(P_\alp,Q_\alp)$ by the set of $H_\alp$ with $H\in  G'$ (where the weak division is computed following the same steps as the one of $S(P,Q)$ by $G'$). 
 
 Actually to obtain this property, we assume that at each step the set $G'$ is ordered in a deterministic way, and that elements $P_1\in G'$ are tested in this order so that the first one such that $\lmon{P_1}$ divides $\lmon{S}$ is used in line $4b(iv)$. In the same way we assume that the pairs $P,Q\in G'$ are taken in a deterministic order  at line $4b$. Then Algorithm 2 starting with $Q_1(z,\Xsoul)$, \ldots, $Q_\ell(z,\Xsoul)$ follows exactly the same steps as Algorithm 1 over $\L=\K$ starting with $Q_1(\alp,\Xsoul)$, \ldots, $Q_\ell(\alp,\Xsoul)$ -- recall that we assume $\Wend(\alp)\neq0$. In more precise terms, each time a polynomial $P$ is considered in Algorithm 2, the polynomial $P_\alp$ is considered at the same step of Algorithm 1, and we have $\lmon{P_\alp}=\lmon{P}$. 

At the end of Algorithm 2, $G = \{ \Pun (z,\Tsoul) , \ldots, \Ps (z,\Tsoul)\}$ is a \GB basis  of the ideal $\Iti$ of $\K(z)[\Tsoul]$ generated by $I$ because Algorithm 2 follows exactly the same steps as Algorithm 1 over $\L = \K(z)$.  Proposition \ref{propelimGB} (with $\L = \K(z)$ and $i =p$) shows that   the set $\calP$ of those $\Pj(z, \Tsoul)$ which depend only on $T_1$, \ldots, $T_p$ (and not on $T_{p+1}$, \ldots, $T_N$) is a  \GB  basis of $\Iti \cap \K(z)[T_1,\ldots,T_p]   = \Iti\cap \K(z)[\varphi_1,\ldots,\varphi_p]$. In part $(i)$ of Proposition \ref{proputile}, we assume that $I \cap \K[z][\varphi_1,\ldots,\varphi_p]=\{0\}$, so that $ \Iti\cap \K(z)[T_1,\ldots,T_p]=\{0\}$ and $\calP = \emptyset$. Therefore for each $j$ there exists $\asj\in\N^N$ such that $\asjindicei \geq 1$ for at least one $i\in \{p+1,\ldots,N\}$ and the coefficient $\lambda_{\asj}(z)$ of $\Tsoul^{\asj}$ in $\Pj(z,\Tsoul)$ is non-zero. This proves that line $5a$ of Algorithm~2 can be carried out. Moreover, since $\Wend(\alp)\neq 0$ we have $\lambda_{\asj}(\alp)\neq 0$ (using line $5b$), so that $\Pj(\alp, \Tsoul)\not\in\K[T_1,\ldots,T_p]$. 

Now since $\Wend(\alp)\neq0$, lines 1 and 2 of Algorithm 1 over $\L=\K$  run  exactly in the same way as lines 1--4 of Algorithm 2. Therefore   $\Pun(\alp,\Tsoul)$, \ldots, $\Ps(\alp,\Tsoul)$ is the output of lines 1--2 of Algorithm 1: it is a \GB basis of $I_\alp$. 
  Then   Proposition \ref{propelimGB} (with $\L=\K$ and $i=p$) shows that  the set of those  $\Pj(\alp,\Tsoul)$
 which  depend only on $T_1$, \ldots, $T_p$ and not on $T_{p+1}$, \ldots, $T_N$ is a \GB basis of $ I_\alp \cap \K[T_1,\ldots,T_p]$. Now we have seen that this set is empty (namely $\Pj(\alp, \Tsoul)\not\in\K[T_1,\ldots,T_p]$ for any $j$),  so that  $ I_\alp \cap \K[T_1,\ldots,T_p] = \{0\}$. Since $\Im(\chialp) =  \K[T_1,\ldots,T_p] $ we deduce that $J_\alp =  \chialp^{-1} ( I_\alp )$ is equal to $\ker(\chialp)$. Now $\Wend(\alp)\neq 0$ implies $W_0(\alp)\neq 0$, so that the linear forms $\varphi_1(\alp,\Xsoul)$, \ldots, $\varphi_p(\alp,\Xsoul)$ are linearly independent over $\K$: the map $\chialp$ is injective, and $J_\alp =   \{0\}$. This concludes the proof of  Proposition \ref{proputile}.\end{proof}

\section{Examples}\label{sec7}

In this section, we give two different illustrations of our algorithm. The first one illustrates Proposition \ref{propclef}, whereas the second one sheds light on  Proposition \ref{proputile} and Algorithm 2 used in its proof.

\medskip

Consider the $E$-functions $f(z)=e^{-iz}+(z-1)^2e^{z}$ and $f'(z)=-ie^{-iz}+(z^2-1)e^{z}$. We want do determine for which $\alpha \in \Qbar^*$ the numbers $f(\alpha)$ and $f'(\alpha)$ are algebraically dependent. We first observe that $f(z)$ and $f'(z)$ are $\Qbar(z)$-algebraically independent, because $1$ and $-i$ are $\mathbb Q$-linearly independent. To apply Beukers' lifting results, we set $g_1(z)=e^z$ and $g_2(z)=e^{-iz}$ which are such that
$$
\left(
\begin{matrix}
f
\\
f'
\end{matrix}
\right) = \left(
\begin{matrix}
(z-1)^2 & 1
\\
z^2-1 & -i
\end{matrix}
\right)
\left(
\begin{matrix}
g_1
\\
g_2
\end{matrix}
\right),\qquad 
\left(
\begin{matrix}
g_1
\\
g_2
\end{matrix}
\right)' = \left(
\begin{matrix}
1 & 0
\\
0 & -i
\end{matrix}
\right)
\left(
\begin{matrix}
g_1
\\
g_2
\end{matrix}
\right).
$$
We introduce the two $\Qbar[z]$-linearly independent linear forms
$$
\varphi_1(z, X_1, X_2)=(z-1)^2X_1+X_2, \quad \varphi_2(z, X_1, X_2)=(z^2-1)X_1-iX_2.
$$
Since $g_1, g_2$ are $\Qbar(z)$-algebraically independent (which our algorithm would have first determined), the ideal
$$
I:=\{Q\in \Qbar[z,X_1,X_2] \quad \textup{such that} \quad Q(z, g_1(z), g_2(z))\equiv 0\}
$$
is reduced to $\{0\}$. 

Let now $\alpha \in \Qbar^*$ be such that there exists $S\in \Qbar[X,Y]\setminus\{0\}$ such that $S(f(\alpha), f'(\alpha))=0$. By Proposition~\ref{propclef}, this is equivalent to the fact that $S(\varphi_1(\alpha, X_1, X_2), \varphi_2(\alpha, X_1, X_2))=Q(\alpha, X_1, X_2)$ for some $Q\in I$, i.e. that 
$$
S\big((\alpha-1)^2X_1+X_2, (\alpha^2-1)X_1-iX_2\big) \equiv 0
$$
as a polynomial in $\Qbar[ X_1, X_2]$. Hence 
$$
S(f(\alpha), f'(\alpha))=0 \Longleftrightarrow  S\big((\alpha-1)^2X_1+X_2, (\alpha^2-1)X_1-iX_2\big) \equiv 0 \;  \textup{in}\; \Qbar[X_1, X_2].
$$

We now set 
$$
D(\alpha):=\left \vert \begin{matrix} (\alpha-1)^2 &1\\ \alpha^2-1 & -i\end{matrix}\right\vert. 
$$
If on the one hand, $D(\alpha)\neq 0$, the linear forms $(\alpha-1)^2X_1+X_2$ and $(\alpha^2-1)X_1-iX_2$ are $\Qbar$-linearly independent so that $S(X,Y)$ must in fact be identically zero in $\Qbar[X,Y]$. In other words, the numbers $f(\alpha)$ and $f'(\alpha)$ are $\Qbar$-algebraically independent.

If on the other hand $D(\alpha) = 0$, i.e. if $\alpha=1$ or $\alpha=i$, the linear forms $(\alpha-1)^2X_1+X_2$ and $(\alpha^2-1)X_1-iX_2$ are $\Qbar$-linearly dependent. If $\alpha=1$, we must have $S(X_2, -iX_2)\equiv 0$, which means that $S(X,Y)$ is in the principal ideal $(iX+Y)$ of $\Qbar[X,Y]$. If $\alpha=i$, we must have $S(-2i X_1+X_2, -2X_1-iX_2)\equiv 0$, which means that $S(X,Y)$ is again in the principal ideal $(iX+Y)$ of $\Qbar[X,Y]$. In both cases, we can indeed take $S(X,Y)=iX+Y$ because it is readily checked that $f(1)=if'(1)$ and $f(i)=if'(i)$. Note that $f(1)=e^{-i}$ and $f(i)=e+(i-1)^2e^i$ are both transcendental by the Lindemann-Weierstrass Theorem, and this could  also be proved by our algorithm or by that in \cite{AR}.

With the notations of \S\S \ref{sec:reduc23commalg} and \ref{sec6}, we have $N=p=2$ and the polynomial $W_0(z)$ defined in the proof of Proposition \ref{proputile} is equal to $D(z)$. Since $I$ is the zero ideal, it is generated by $F = \emptyset$. Running Algorithm 2  with this input is trivial: the output is $W(z)=W_0(z)=D(z)$. Moreover for each root $\alpha$ of this polynomial, the algorithm computes the linear relation $f(\alpha) = i f' (\alpha)$.

\bigskip

As a second illustration, consider the transcendental $E$-function, of hypergeometric type,
$$
f(z):={}_1F_2\left[\begin{matrix} 1/2\\
1/3, 2/3 \end{matrix};z^2\right] = \sum_{n=0}^\infty \frac{(2n)!}{n!(3n)!} \left(\frac{27z^2}4\right)^{n}.
$$
It is a solution of the differential equation (of minimal order for $f$)
\begin{equation}\label{eq:JAW}
9z^2y'''(z)+9zy''(z)-(36z^2+1)y'(z)-36zy(z)=0.
\end{equation}
A basis of local solutions at $z=0$ of \eqref{eq:JAW} is given by
$$
f(z), \quad z^{2/3}\, {}_1F_2\left[\begin{matrix} 5/6\\
2/3, 4/3 \end{matrix};z^2\right], \quad z^{4/3}\, {}_1F_2\left[\begin{matrix} 7/6\\
4/3, 5/3 \end{matrix};z^2\right].
$$
Using the algorithm in \cite{bcvw}, J.-A. Weil confirmed to us that the  
 differential Galois group~(\footnote{The possible differential Galois groups of hypergeometric equations  are classified in \cite{katzlivre}.}) of \eqref{eq:JAW} is $SO(3,\mathbb C)$ and the ideal of polynomial relations in $\Qbar[z][X_1, X_2, X_3]$ between $f(z), f'(z), f''(z)$ is principal, generated by the first integral
\begin{equation}\label{eq:JAW2}
f(z)^2-\frac14f'(z)^2+\frac{9z^2}{4} \big(4f(z)-f''(z)\big)^2 = 1.
\end{equation} 
In particular, $f(z)$ and $f'(z)$ are $\Qbar(z)$-algebraically independent. Let us prove, following our algorithm, that $\alpha=0$ is the only algebraic number such that $f(\alpha)$ and $f'(\alpha)$ are algebraically  dependent over $\Qbar$. 

We assume that Feng's algorithm provides the generator 
$$Q(z,\Xsoul) = X_1^2 - \frac14 X_2^2 + \frac{9z^2}4 (4X_1-X_3)^2 - 1$$
of the ideal
$$I = \{T\in\Qbar[z,X_1,X_2,X_3], \, T(z,f(z),f'(z),f''(z))=0\}.$$
Let us follow Algorithm 2 that appears in the proof of Proposition \ref{proputile}  (see \S \ref{sec6}), with $N=3$, $p=2$, $\varphi_1(z,\Xsoul) = X_1$,  $\varphi_2(z,\Xsoul) = X_2$, and $T_i=X_i$. The input is $F=\{Q\}$ and $W_0=1$. The second elimination order is such that $X_1^2 > X_2^2 > X_1X_3 > X_3^2$ so that $\lcoeff{Q}=1+36z^2$. After Step 3 of Algorithm 2 we have $W(z) = z^2+\frac1{36}$ by choosing the monic least common multiple. Then $W$ does not change at Step 4 since $G=F=\{Q\}$ contains only one element. In Step 5$a$ we may choose $\asoul = (0,0,2)$; then after line 5$b$ we have $W(z) = z^2( z^2+\frac1{36})$. This is the polynomial we obtain such that property $(i)$ of  Proposition \ref{proputile}  holds. Therefore if $\alpha\in\Qbar\etoile$ is such that  that $f(\alpha)$ and $f'(\alpha)$ are algebraically  dependent over $\Qbar$, then $\alpha$ is a root of $W$:  $\alpha=\pm i/6$. 

Now we follow the proof of  part $(ii)$ of  Proposition \ref{proputile} to determine, for each of these values of  $\alpha$, whether $f(\alpha)$ and $f'(\alpha)$ are algebraically  dependent or not.  For $\alpha=\pm i/6$, the ideal $I_\alpha$ defined after the statement of Proposition \ref{proputile} is generated by
$$Q_\alpha(\Xsoul) =Q(\alpha,\Xsoul) =  - \frac14 X_2^2- \frac1{16}X_3^2 + \frac12  X_1 X_3 - 1.$$
Since $Q_\alpha\not\in\Qbar[T_1,T_2]=\Qbar[X_1,X_2] $ we obtain that the empty set is a \GB basis of  the ideal $J_\alpha$ defined in \S \ref{sec6}, so that it is equal to $\{0\}$: $f(\alpha)$ and $f'(\alpha)$ are algebraically  independent over $\Qbar$. This concludes the proof that $0$ is the only algebraic number $\alpha$ such that $f(\alpha)$ and $f'(\alpha)$ are algebraically  dependent.

We point out that the apparently exceptional points $ \alpha=\pm i/6$ have appeared in the above computation because the leading monomial of 
$Q_\alpha(\Xsoul)$ is  equal to  $X_2^2$ for these values, whereas it is $X_1^2$ for $\alpha\neq\pm i/6$. This illustrates the general situation in the proof of  part $(i)$ of  Proposition \ref{proputile}: all agebraic points $\alpha$ where the computations take place in a non-generic way are gathered in the set of roots of the polynomial $W$; then part $(ii)$ enables one, for each such $\alpha$, to see whether the $E$-functions under consideration take algebraically dependent values at $\alpha$ or not.

\medskip

\noindent S. Fischler, Laboratoire de Math\'ematiques d'Orsay, Univ. Paris-Sud, CNRS, Universit\'e 
Paris-Saclay, 91405 Orsay, France.

\medskip

\noindent T. Rivoal, Institut Fourier, CNRS et Universit\'e Grenoble Alpes, CS 40700, 38058 Grenoble cedex 9, France.

\medskip

\noindent Keywords: $E$-functions, Algebraic independence, Differential equations, \GB bases, 
Algorithm.

\medskip

\noindent MSC 2000: 11J91, 13P10, 33E30, 34M05.

\end{document}